\documentclass[11pt,leqno]{amsart}%
\usepackage{amsmath}
\usepackage{amsfonts}
\usepackage{amssymb}
\usepackage{graphicx}%
\usepackage[usenames]{color}
\usepackage{hyperref}
\providecommand{\U}[1]{\protect\rule{.1in}{.1in}}
\newtheorem{theorem}{Theorem}

\newtheorem{corollary}[theorem]{Corollary}

\newtheorem{example}[theorem]{Example}

\newtheorem{lemma}[theorem]{Lemma}

\newtheorem{proposition}[theorem]{Proposition}
\newtheorem{remark}[theorem]{Remark}

\numberwithin{equation}{section}

\begin{document}

\title[Complete self--shrinkers into some regions of the space]{Complete self--shrinkers confined into some regions of the space}

\begin{abstract}
We study geometric properties of complete non--compact bounded self--shrinkers and obtain natural restrictions that force these hypersurfaces to be compact. Furthermore, we observe that, to a certain extent, complete self--shrinkers intersect transversally a hyperplane through the origin. When such an intersection is compact, we deduce spectral information on the natural drifted Laplacian associated to the self--shrinker. These results go in the direction of verifying the validity of a conjecture by H. D. Cao concerning the polynomial volume growth of complete self--shrinkers. A finite strong maximum principle in case the self--shrinker is confined into a cylindrical product is also presented.
\end{abstract}

\date{\today}

\author {Stefano Pigola}
\address{Dipartimento di Scienza e Alta Tecnologia\\
Universit\`a degli Studi dell'Insubria\\
via Valleggio 11\\
I-22100 Como, ITALY}
\email{stefano.pigola@uninsubria.it}

\author{Michele Rimoldi}
\address{Dipartimento di Scienza e Alta Tecnologia\\
Universit\`a degli Studi dell'Insubria\\
via Valleggio 11\\
I-22100 Como, ITALY}
\email{michele.rimoldi@gmail.com}

\subjclass[2010]{53C21}

\keywords{Bounded self--shrinkers, hyperplane intersection, weighted manifolds, drifted Laplacian}

\maketitle

\tableofcontents

\section*{Introduction}

By a self shrinker \textquotedblleft based at\textquotedblright\ $x_{0}%
\in\mathbb{R}^{m+1}$ we mean a connected, isometrically immersed hypersurface
$x\colon \Sigma^{m}\rightarrow\mathbb{R}^{m+1}$ whose mean curvature vector field
$\mathbf{H}$ satisfies the equation%
\[
\left(  x-x_{0}\right)  ^{\bot}=-\mathbf{H},
\]
where $\left(  \cdot\right)  ^{\bot}$ denotes the projection on the normal bundle of $\Sigma$.
Note that we are using the convention
\[
\mathbf{H}=\mathrm{tr}_{\Sigma}\mathbf{A},
\]
where the second fundamental form of the immersion is defined as the generalized Hessian
\[
\mathbf{A}=Ddx.
\]
With this convention, if $\Sigma$ is oriented by the outer unit normal $\nu$ and we let
\[
\mathbf{H}=H \nu,
\]
then $\Sigma$ is mean--convex provided $H\leq 0$ and, furthermore, the self--shrinker equation takes the scalar form
\[
\left\langle x-x_0 , \nu \right \rangle= -H.
\]

In this paper we shall consider only self--shrinkers based at $0\in \mathbb{R}^{m+1}$.
Natural examples of complete, properly embedded self-shrinkers are the cylindrical products
\begin{equation}\label{cylinders}
\mathcal{C}^{k,m-k}_{\sqrt{k}}=\mathbb{S}^{k}_{\sqrt{k}}\times \mathbb{R}^{m-k}, \text{ }k=0,...,m,
\end{equation}
which include, as extreme cases, the sphere $\mathbb{S}^{m}_{\sqrt{m}}$ and all the hyperplanes through the origin of $\mathbb{R}^{m+1}$.
Actually, according to a classification theorem by T. Colding and W. Minicozzi, \cite{CoMi-Annals2012}, these are the only complete, embedded
and mean-convex self-shrinkers with extrinsic polynomial volume growth, i.e.,
\[
\mathrm{vol}(\mathbb{B}^{m+1}_R \cap \Sigma) \leq CR^n
\]
for some $C>0$, $n \in \mathbb{N}$ and for every $R>>1$; here $\mathbb{B}^{m+1}_R$ denotes the ball in the ambient Euclidean space.

We stress that it was conjectured by H.-D. Cao, \cite{CaLi-CalcVar}, that every complete self--shrinker has extrinsic polynomial (Euclidean, in fact) volume growth. By a very interesting result due to X. Cheng and D. Zhou, \cite{ChZh-volume}, that completes a previous theorem by Q. Ding and Y.L. Xin, \cite{DiXi-proper}, this is equivalent to the fact that the immersion is proper. Thus, by way of example, if Cao Conjecture was true, then any complete self--shrinker in a ball of $\mathbb{R}^{m+1}$ should be compact. In order to obtain indications on the validity of this conjecture, it is then relevant to understand which geometric constraints are imposed by the assumption that a complete self--shrinker is bounded and to obtain natural and general restrictions that force these hypersurfaces to be compact. For instance, we will prove the following results.

\begin{theorem}\label{intro-th_bounded}
Let $x:\Sigma^{m}\rightarrow\mathbb{B}^{m+1}_{R_0}(0) \subset \mathbb{R}^{m+1}$ be a complete self--shrinker.
\begin{itemize}
\item[(a)] Assume $|\mathbf{A}| \leq 1$. Then: \medskip
\begin{itemize}
\item[(a.1)] $R_{0}\geq\sup_{\Sigma}\left\vert \mathbf{H}\right\vert=\sqrt{m}$.
\item[(a.2)] If $m=2$, then $\Sigma=\mathbb{S}^{2}_{\sqrt{2}}$.
\item[(a.3)] If $m\geq 3$ and $\Sigma$ is non-compact, then $\Sigma$ must be connected at infinity, i.e., it has only one end. Moreover, $|\mathbf{A}|<1$, the universal cover $\tilde\Sigma$ enjoys the loops to infinity property along every ray, \cite{So-Indiana}, and every f.g. subgroup of the fundamental group of $\Sigma$ grows at most polynomially of order $m$.
\end{itemize}
\item[(b)]Assume $\lim_{R\to\infty}\sup_{\Sigma\setminus B_R^{\Sigma}}|\mathbf{A}|<1$. Then $\Sigma$ is compact.
\item[(c)] Assume  $|\mathbf{A}| \in L^{p}(\Sigma)$, for some $p\geq m$. Then $\Sigma$ is compact.

\end{itemize}

\end{theorem}

More generally, one can try to understand the geometry of self--shrinkers which are confined in a connected region bounded by some dilated cylinder $\mathcal{C}^{k,m-k}_{R}$, $R\geq \sqrt{k}$. In this setting, as a preliminary and simple fact, we observe the validity of the following (finite) strong maximum principle.

\begin{theorem}
Let $x\colon \Sigma^{m}\rightarrow\mathbb{R}^{m+1}$ be a complete self-shrinker. Assume that
$\left\vert \mathbf{H}\right\vert \leq\sqrt{k}$ and that $x\left(  \Sigma\right)  $ is confined inside the domain
bounded by $\mathcal{C}^{k,m-k}_{R}$.
If $x\left(  \Sigma\right)  \cap \mathcal{C}^{k,m-k}_{R} \neq\emptyset$ then:
\begin{itemize}
\item[(a)] $R=\sqrt{k}$,
\item[(b)] $x:\Sigma \rightarrow\mathcal{C}^{k,m-k}_{\sqrt{k}}$ is a Riemannian covering map. In particular, if
$k\geq2$, then $\Sigma=\mathcal{C}^{k,m-k}_{\sqrt{k}}$ in the Riemannian sense.
\end{itemize}
\end{theorem}

Actually, when $k=0$ and, hence, $\mathcal{C}^{0,m}$ is a hyperplane through the origin, it is reasonable to expect that
the self--shrinker cannot be located into one of the corresponding half-spaces. We are able to verify that, to a certain extent, this is in fact true.
The next result can be considered as a weak half--space theorem for complete self--shrinkers.
\begin{theorem}
Let $x\colon \Sigma^{m}\rightarrow\mathbb{R}^{m+1}$ be
a complete, self--shrinker. Assume that either one of the following assumptions is satisfied:
\begin{itemize}
\item[(a)]$\Sigma$ has extrinsic polynomial volume growth (equivalently, $\Sigma$ is properly immersed).

\item[(b)] $\left\vert \mathbf{A}\right\vert ^2 \in L^{p} \left(
d\mathrm{vol}_{f}\right)$ with $\left\vert \mathbf{A}\right\vert ^{2}\leq 1+\frac{1}{p}$,
for some $p> 1$.
\end{itemize}
Then, for every hyperplane $\Pi$ through the origin of $\mathbb{R}^{m+1}$, $\Sigma$ cannot be contained
in one of the closed half--spaces determined by $\Pi$ unless $\Sigma=\Pi$.
\end{theorem}

Accordingly, and in view of the strong maximum principle, it is also reasonable to assume that some transversal intersection between a self--shrinker and a hyperplane through the origin occurs. When such an intersection is compact, we can obtain information on the spectrum of the natural drifted Laplacian $\Delta_{f}=\Delta - \left \langle \nabla, \nabla f \right \rangle$, with $f=|x|^2/2$.

\begin{theorem}
Let $i:\Sigma^{m}\hookrightarrow\mathbb{R}^{m+1}$ be a \ complete, embedded
self-shrinker. Assume that, for some hyperplane $\Pi\approx\mathbb{R}^{m}$
through the origin, $\Sigma\cap\Pi=K$ is a compact $\left(  m-1\right)
$-dimensional submanifold. Then:

\begin{enumerate}
\item[(a)] for every connected component $\Sigma_{1}$ of $\Sigma\backslash K$
(which is an open submanifold $\Sigma_{1}\subset\Sigma$ with $\partial
\Sigma_{1}\subseteq K$) it holds $\lambda_{1}(  -\Delta_{f}^{\Sigma_{1}})  \geq1.$

\item[(b)] If either $\Sigma$ is compact or $\Sigma$ has only one end, then
there exists a compact connected component $\Sigma_{2}$ of $\Sigma\backslash
K$ such that $\lambda_{1}(  -\Delta_{f}^{\Sigma_{2}})  =1.$

\item[(c)] If $\Sigma_{3}$ is an end of $\Sigma$ with respect to $K$ and%
\[
\mathrm{vol}\left(  \Sigma_{3}\cap \mathbb{B}_{R}^{m+1}\right)
=O(  e^{\alpha R^{2}})  \text{, as }R\rightarrow+\infty,
\]
for some $0\leq\alpha<1/2$, then $\lambda_{1}(  -\Delta_{f}^{\Sigma_{3}})  =1.$

\end{enumerate}
\end{theorem}

\bigskip

\noindent \textbf{Acknowledgments.} The authors would like to thank Pacelli Bessa, Debora Impera and Giona Veronelli for their interest in this work and for several suggestions that have improved the presentation of the paper.
\section{Some notations}

Throughout the paper we let%
\[
f=\frac{\left\vert x\right\vert ^{2}}{2}%
\]
and we denote by $d\mathrm{vol}_{f}$ the corresponding weighted volume measure of $\Sigma$, i.e.,%
\[
d\mathrm{vol}_f=e^{-f}d\mathrm{vol}.
\]
Thus, $\Sigma_{f}=\left(  \Sigma,g,d\mathrm{vol}_{f}\right)  $ is a smooth metric measure
space. The weighted measure of the intrinsic geodesic ball $B_{R}^{\Sigma}(o)=\{p\in \Sigma:d_\Sigma(o,p)<R\}$ is
given by%
\[
\mathrm{vol}_{f}\left(  B_{R}^{\Sigma}\right)  =\int_{B_{R}^{\Sigma}}d\mathrm{vol}_{f}.
\]
Note that, obviously,
\[
\mathrm{vol}_{f}\left(  B_{R}^{\Sigma}(o)\right) \leq \mathrm{vol}_{f}\left( \mathbb{B}_{R}(x(o))\cap \Sigma \right),
\]
where $\mathbb{B}_{R}$ denotes the Euclidean ball.

There is a natural drifted Laplacian on $\Sigma_{f}$ defined by%
\[
\Delta_{f}=e^{f}\operatorname{div}\left(  e^{-f}\nabla\right)  =\Delta
-\left\langle \nabla,\nabla f\right\rangle .
\]
It is symmetric on $L^{2}\left(  d\mathrm{vol}_{f}\right)  $ and it can be expressed in the equivalent form%
\[
\Delta_{x^{T}}=\Delta-\left\langle \nabla,x^{T}\right\rangle ,
\]
where $x^{T}$ denotes the tangential component of the immersion.

Recall also that the Bakry-Emery Ricci tensor of $\Sigma_{f}$ is defined by%
\[
\mathrm{Ric}_{f}=\mathrm{Ric}+\mathrm{Hess}\left(  f\right)  .
\]
Using once again the self--shrinker equation we easily obtain the following
very important estimate,  \cite{Ri-self},%
\begin{equation}
\mathrm{Ric}_{f}\geq1-\left\vert \mathbf{A}\right\vert ^{2}\label{ric_f}%
\end{equation}
where $\mathbf{A}$ denotes the second fundamental tensor of the immersion
$x\colon \Sigma^{m}\rightarrow\mathbb{R}^{m+1}$. Indeed, by Gauss equations,%
\[
\mathrm{Ric}\geq\left\langle \mathbf{H},\mathbf{A}\right\rangle -\left\vert
\mathbf{A}\right\vert ^{2}g,
\]
whereas, by the self-shrinker equation,%
\[
\mathrm{Hess}(f)
=g+\left\langle x^{\bot},\mathbf{A}\right\rangle =g-\left\langle
\mathbf{H},\mathbf{A}\right\rangle .
\]

\section{A maximum principle} \label{Section maxprinc}

To begin with, we observe that if a complete self--shrinker with $|\mathbf{A}|\leq 1$ is contained in a ball and it is tangent to the boundary of this ball at a point, then it must be the standard sphere $\mathbb{S}^{m}_{\sqrt{m}}$. The analytic proof is a straightforward application of the maximum principle for subharmonic functions. Later on, in Section \ref{Section A1}, we shall come back on this kind of arguments.

\begin{proposition}\label{prop_maxprinc_sphere}
Let $x:\Sigma^m\to\mathbb{R}^{m+1}$ be a complete bounded self--shrinker with $|\mathbf{H}|\leq \sqrt{m}$. If there exist $x_0\in \Sigma$ such that $|x|(x_0)=\sup_{\Sigma}|x|$, then $|x|\equiv\sqrt{m}$ and $\Sigma$ is the standard sphere $\mathbb{S}^{m}_{\sqrt{m}}$.
\end{proposition}

\begin{proof}
Recall that, \cite{CoMi-Annals2012},
\begin{equation}\label{LaplNormImm}
\Delta|x|^{2}=2\left(  m-\left\vert \mathbf{H}\right\vert
^{2}\right)  ,
\end{equation}
therefore, by assumption,
\[
\ \Delta|x|^2\geq0.
\]
Using the strong maximum principle we thus obtain $|x|\equiv c>0$. This implies that $x:\Sigma^m \to \mathbb{S}^m_{c}$ is a Riemannian covering projection, hence an isometry since $\mathbb{S}^{m}_{c}$ is simply connected. In particular, by the self--shrinker equation, $c=\sqrt{m}$.
\end{proof}
The above result can be deduced more geometrically via a suitable application of the usual touching principle. We adopt this viewpoint to obtain the following strong maximum principle for self-shrinkers. Recall that the oriented hypersurface $x:\Sigma^m \rightarrow\mathbb{R}^{m+1}$ is called mean--convex at $p\in\Sigma$ if $\mathbf{H}(p)=H(p)\nu$ where $H(p)\leq0$ and $\nu$ is the outward pointing unit normal at $p$.

\begin{theorem}[Maximum principle]\label{th_maxprinc}
Let $\Omega\subset\mathbb{R}^{m+1}$ be a domain such that
$i:\partial\Omega\hookrightarrow\mathbb{R}^{m+1}$ is a properly embedded
self-shrinker. Let $x\colon \Sigma^{m}\rightarrow\mathbb{R}^{m+1}$ be a complete
self-shrinker satisfying $x\left(  \Sigma\right)  \subseteq\overline
{\Omega_{\lambda}}$ \ for some $\lambda>0$, where $\Omega_{\lambda}%
=\lambda\Omega$ denotes the $\lambda$-dilation of $\Omega$. \ Assume that
$x\left(  \Sigma\right)  \cap\partial\Omega_{\lambda}\neq\emptyset$ and that, for each intersection point $x(p)$,
there exist a neighborhood $V\subset \mathbb{R}^{m+1}$ of $x(p)$ and a neighborhood $W \subset \Sigma$ of $p$ such that:
\begin{itemize}
\item[(i)]  $\partial \Omega\cap \lambda^{-1}V$ is mean convex
\item[(ii)] $\sup_{W}\left\vert \mathbf{H}_{\Sigma}\right\vert \leq\inf_{\lambda^{-1}V \cap \partial\Omega} \left\vert \mathbf{H}_{\partial\Omega}\right\vert$.
\end{itemize}
Then
\begin{itemize}
\item[(a)] $\lambda=1$,
\item[(b)] $\partial \Omega = S^{k}_{\sqrt{k}} \times \mathbb{R}^{m-k}$, for some $k\in\{0,...,m\}$,
\item[(c)] $x:\Sigma\rightarrow\partial\Omega$ is a Riemannian covering map.
\end{itemize}
In particular, if $\partial \Omega$ is simply connected (e.g. if $\,k \geq 2$ in (b)), then $\Sigma=\partial\Omega$ in
the Riemannian sense.
\end{theorem}

A situation of special interest is obtained by choosing $\partial \Omega$ to be a cylindrical product shrinker $\mathcal{C}^{k,m-k}_{\sqrt{k}}$. Note that
the case $k=m$ is precisely the content of Proposition \ref{prop_maxprinc_sphere}.
\begin{corollary}
Let $x\colon \Sigma^{m}\rightarrow\mathbb{R}^{m+1}$ be a complete self-shrinker. Assume that
$\left\vert \mathbf{H}_{\Sigma}\right\vert \leq\sqrt
{k}$ and that $x\left(  \Sigma\right)  $ is confined inside the solid
cylinder bounded by $\mathcal{C}^{k,m-k}_{R}=$ $\mathbb{S}_{R}^{k}\times\mathbb{R}^{m-k}$.
If $x\left(  \Sigma\right)  \cap\mathcal{C}^{k,m-k}_{R}\neq\emptyset$ then
\begin{itemize}
\item[(a)] $R=\sqrt{k}$,
\item[(b)] $x:\Sigma \rightarrow\mathcal{C}^{k,m-k}_{\sqrt{k}}$ is a Riemannian covering map.
\end{itemize}
In particular, if $k\geq2$, then $\Sigma=\mathcal{C}^{k,m-k}_{\sqrt{k}}$.
\end{corollary}

\begin{proof}
[Proof (of Theorem \ref{th_maxprinc})]Let
\[
\mathcal{O}=x^{-1}\left(  \partial\Omega_{\lambda}\right)  .
\]
Since $x$ is smooth and $\partial\Omega_{\lambda}$ is closed in $\mathbb{R}%
^{m+1}$, we have that $\mathcal{O}$ is a closed subset of $\Sigma$. We claim
that $\mathcal{O}$ is also open so that, by a connectedness argument,
$\mathcal{O}=\Sigma$ i.e. $x\left(  \Sigma\right)  \subseteq\partial
\Omega_{\lambda}$. To this end, let $p\in\mathcal{O}$. Observe that, by the
mean-convexity assumption (i), in a connected neighborhood
$\lambda^{-1}U_{x(p)}\subset\partial\Omega$  it holds
\[
\mathbf{H}_{\partial\Omega}=H_{\partial\Omega}\nu_{\partial\Omega},
\]
where $H_{\partial\Omega}\leq0$ and $\nu_{\partial\Omega}$ denotes the
exterior pointing unit normal to $\partial\Omega$. Moreover, the rescaling
property of the mean curvature tells us that
\[
H_{\partial\Omega_{\lambda}}(x(p))=\lambda^{-1}H_{\partial\Omega}(\lambda
^{-1}x(p)).
\]
Whence, using the fact that $i:\partial\Omega\hookrightarrow\mathbb{R}^{m+1}$
is a self-shrinker, it is standard to deduce that either $H_{\partial
\Omega_{\lambda}}\equiv0$ in $U_{x(p)}$, or $H_{\partial\Omega_{\lambda}}<0$ on $U_{x(p)}$; see e.g. the
beginning of the proof of \cite[Theorem 2]{Ri-self}. In the first case, by
assumption, we must have $\mathbf{H}_{\Sigma}=0$ in a neighborhood of $p$ in
$\Sigma$ and the result reduces to a well known local maximum principle for
minimal surfaces. Therefore, from now on, we assume%
\[
H_{\partial\Omega_{\lambda}}<0\text{ in }U_{x(p)}.
\]

Since $x\left(  \Sigma\right)  $ lies inside $\overline{\Omega_{\lambda}}$,
then $x\left(  \Sigma\right)  $ must intersect $\partial\Omega_{\lambda}$
tangentially at $p\in\mathcal{O}$ and
\[
\nu_{\Sigma}\left(  p\right)  =\nu_{\partial\Omega_{\lambda}}\left(  x\left(
p\right)  \right)
\]
the outward pointing unit normal to $\Omega_{\lambda}$. It follows from the
self-shrinker equations for $\partial\Omega$ and $\Sigma$, and the rescaling property of the
mean curvature, that
\[
\mathbf{H}_{\Sigma}\left(  p\right)  =\lambda^{2}\mathbf{H}_{\partial
\Omega_{\lambda}}(x(p))=\lambda\mathbf{H}_{\partial\Omega}(\lambda^{-1}x(p)).
\]
Combining this latter with assumption (ii) we get%
\[
\lambda^{2}\left\vert \mathbf{H}_{\partial\Omega_{\lambda}}\left(  x\left(
p\right)  \right)  \right\vert =\left\vert \mathbf{H}_{\Sigma}\left(
p\right)  \right\vert \leq\left\vert \mathbf{H}_{\partial\Omega}\left(
\lambda^{-1}x\left(  p\right)  \right)  \right\vert =\lambda\left\vert
\mathbf{H}_{\partial\Omega_{\lambda}}\left(  x\left(  p\right)  \right)
\right\vert .
\]
Thus%
\[
\lambda\leq1.
\]
If we write, in a neighborhood of $p$:
\[
\mathbf{H}_{\Sigma}=H_{\Sigma}\nu_{\Sigma}\text{ and }\mathbf{H}%
_{\partial\Omega_{\lambda}}=H_{\partial\Omega_{\lambda}}\nu_{\partial
\Omega_{\lambda}},
\]
then, by mean convexity of $U_{x(p)}$, by the above equation at $p$, and by
continuity, we have,  in a neighborhood of $p$,
\[
H_{\Sigma}\text{, }H_{\partial\Omega_{\lambda}}(x)<0
\]
and
\[
H_{\Sigma}\geq H_{\partial\Omega}\left(  \lambda^{-1}x\right)  =\lambda
H_{\partial\Omega_{\lambda}}\left(  x\right)  \geq H_{\partial\Omega_{\lambda
}}\left(  x\right)  .
\]

We can now apply the usual touching principle and deduce that, actually,
$x(\Sigma)$ and $\partial\Omega_{\lambda}$ coincide in a small neighborhood of $p$. This
proves the claim and, as already remarked at the beginning of the proof,
$x\left(  \Sigma\right)  \subseteq\partial\Omega_{\lambda}$.

Now,
$x:\Sigma\rightarrow\partial\Omega_{\lambda}$ is a local isometry between
complete manifolds, hence, it is a covering map. In particular, $x\left(  \Sigma\right) =\partial\Omega_{\lambda}$, and from the equality

\[
H_{\Sigma}\left(  p\right)  =\lambda^{2}H_{\partial\Omega_{\lambda}}\left(
x\left(  p\right)  \right)
\]
we deduce%
\[
H_{\partial\Omega_{\lambda}}\left(  x\right)  =H_{\Sigma
} =\lambda^{2}H_{\partial\Omega_{\lambda}}\left(  x\right)  ,
\]
that is
\[
\lambda=1.
\]
This shows that $x(\Sigma)=\partial\Omega$. Finally, by assumption (i),
$\partial\Omega$ is a properly embedded self-shrinker satisfying
$H_{\partial\Omega}\leq0$ everywhere. Since properly immersed self--shrinkers
have polynomial (actually Euclidean) volume growth, \cite{DiXi-proper,
ChZh-volume}, to complete the proof we apply a classification result by T.
Colding and W. Minicozzi, \cite[Theorem 0.17]{CoMi-Annals2012}.
\end{proof}

\section{Self--shrinkers in a ball}\label{section ball}

The aim of this section is to show that certain  boundedness conditions on the norm of the second fundamental form prevent the existence of complete, non--compact, bounded self--shrinkers.

\subsection{Estimate of the exterior radius}

The sphere $\mathbb{S}_{\sqrt{m}}^{m}$ is a self--shrinker of  constant mean curvature $-\sqrt{m}$ and contained
in the compact ball $\overline{\mathbb{B}}_{\sqrt{m}}^{m+1}\left(  0\right)
$. Our first remark is that if a complete self--shrinker with controlled intrinsic volume growth is contained in some
ball $\mathbb{B}_{R_{0}}^{m+1}\left(  0\right)  ,$ then there is an obvious
relation between the ray $R_{0}$ and the dimension $m$.

\begin{proposition}\label{PropBall}
Let $x\colon \Sigma^{m}\rightarrow\mathbb{R}^{m+1}$ be a complete non--compact self--shrinker
whose intrinsic volume growth satisfies
\[
R \to \frac{R}{\log\mathrm{vol}\left(B_{R}^{\Sigma}\right)} \not\in L^1(+\infty).
\]
If $x\left(  \Sigma\right)
\subseteq\overline{\mathbb{B}}_{R_{0}}^{\,m+1}\left(  0\right)$,
then%
\[
R_{0}\geq\sup_{\Sigma}\left\vert \mathbf{H}\right\vert \geq\sqrt{m}.
\]

\end{proposition}

\begin{proof}
Recall that, by the self--shrinker equation,%
\[
\Delta_{f}|x| ^{2}=2\left(  m-|x|^{2}\right)  .
\]
On the other hand, since
\[
c^{-1}d\mathrm{vol}_{f}\leq d\mathrm{vol}\leq cd\mathrm{vol}_{f}%
\]
for a large enough constant $c>1$, then
\[
\frac{R}{\log\mathrm{vol}_{f}\left(  B_{R}^{\Sigma}\right)} \not\in L^1(+\infty)
\]
and this implies that the weighted manifold $\Sigma_{f}$ enjoys the weak
maximum principle at infinity for the drifted Laplacian $\Delta_{f}$, \cite{PiRiSe-Memeoirs, PiRimSe-MathZ}.
Therefore%
\[
0\geq2\left(  m-\sup_{\Sigma}|x|^{2}\right)  \geq2\left(
m-R_{0}^{2}\right)  ,
\]
and the claimed lower estimate on $R_{0}$ follows. Now, from the self--shrinker
equation we have%
\[
\sup_{\Sigma}\left\vert \mathbf{H}\right\vert \leq|x|\leq
R_{0}.
\]
Using this information into equation (\ref{LaplNormImm}):
\begin{equation}\nonumber
\Delta|x|^{2}=2\left(  m-\left\vert \mathbf{H}\right\vert
^{2}\right)  ,
\end{equation}
and noting also that the weak maximum principle at infinity for the Laplacian
holds on $\Sigma$, we deduce%
\[
0\geq2\left(  m-\sup_{\Sigma}\left\vert \mathbf{H}\right\vert ^{2}\right)  .
\]
This completes the proof.
\end{proof}

\begin{remark}
\rm{In particular, if $\Sigma$ has extrinsic polynomial volume growth, then the above radius estimate holds. This follows from the obvious relation

\[
\mathrm{vol}(B^{\Sigma}_R)\leq \mathrm{vol}(\mathbb{B}^{m+1}_R \cap \Sigma).
\]}
\end{remark}

Note that, by \cite[Theorem 2.2]{Y} and inequality \eqref{ric_f}, a complete non--compact bounded self--shrinker $x:\Sigma^m\to\mathbb{R}^{m+1}$ with $|\mathbf{A}|\leq1$ satisfies the sharp estimate
\begin{equation}\label{IntrVolGrAleq1}
\mathrm{vol}(B_R^{\Sigma})\leq C R^m.
\end{equation}
Moreover, since $|\mathbf{A}|\leq1$, by the Cauchy--Schwarz inequality we have that $|\mathbf{H}|^2\leq m$.
We can hence specialize Proposition \ref{PropBall} to the following
\begin{corollary}\label{CoroBall}
Let $x:\Sigma^m\to\mathbb{R}^{m+1}$ be a complete non--compact self--shrinker with $|\mathbf{A}|\leq 1$. If $x(\Sigma)\subseteq \overline{\mathbb{B}}_{R_{0}}^{\,m+1}\left(  0\right)$, then
\[
\ R_{0}\geq\sup_{\Sigma}|\mathbf{H}|=\sqrt{m}.
\]
\end{corollary}

\subsection{Bounded self--shrinkers with $|\mathbf{A}|\leq 1$} \label{Section A1}
As a consequence of the strong maximum principle for the Laplace-Beltrami operator, we observed in Section \ref{Section maxprinc} that, for a self--shrinker satisfying $|\mathbf{A}|\leq 1$, hence $|\mathbf{H}|\leq \sqrt{m}$, the norm of the immersion cannot attain a finite maximum unless the shrinker is a round sphere of radius $\sqrt{m}$. In particular, this applies to any compact self--shrinker with the same bound on the mean curvature. It is by now well understood that parabolicity is a good substitute of compactness. For two-dimensional shrinkers this property is implied by the above condition on the second fundamental form.

\begin{theorem}
Let $x:\Sigma^2\to\mathbb{R}^3$ be a complete bounded self--shrinker with $|\mathbf{A}|\leq1$. Then $\Sigma=\mathbb{S}^{2}_{\sqrt{2}}$.
\end{theorem}

\begin{proof}
Since $m=2$, we know from \eqref{IntrVolGrAleq1} that $\Sigma$ has quadratic intrinsic volume growth, therefore it is parabolic (possibly compact); see e.g. \cite{Gr-BAMS}. As in Proposition \ref{prop_maxprinc_sphere}, since $|\mathbf{H}|\leq\sqrt{2}$,  $|x|^2$ is a bounded subharmonic function and we obtain that $|x|\equiv const$. This implies $\Sigma=\mathbb{S}^{2}_{\sqrt{2}}$.
\end{proof}

In higher dimensions, the same control gives information on the topology at infinity of a bounded shrinker.

\begin{theorem}
\label{OneEnd} Let $x:\Sigma^{m}\to\mathbb{R}^{m+1}$ be a complete
non--compact bounded self--shrinker with $|\mathbf{A}|\leq1$. Then $\Sigma$
does not contain a line. In particular, $\Sigma$ is connected at infinity,
i.e., $\Sigma$ has only one end.
\end{theorem}

\begin{remark}
\rm{ Applying this result to the universal covering of $\Sigma$, and
using \cite{So-Indiana, WeWa-JDG, Y}, we also get the topological information
collected in Theorem \ref{intro-th_bounded} stated in the Introduction.}
\end{remark}

\begin{proof}
Assume by contradiction that $\Sigma$ contains a line. By assumption
and \eqref{ric_f}, we have that $Ric_{f}\geq0$ with $f$ bounded. Therefore, we can apply the
Cheeger--Gromoll--Lichnerowicz splitting theorem, \cite{L}, and obtain that
$\Sigma$ splits isometrically as the Riemannian product $\left(  N^{m-1}%
\times\mathbb{R},g_{N}+dt\otimes dt\right)  $. Moreover $f$ is constant along the
line. Thus

\begin{equation}
\mathrm{Hess}(f)(\partial_{t},\partial_{t})=0.\label{splitting-hessian}%
\end{equation}
On the other hand, consider the Simons type equation, see e.g. \cite{Hu-JDG, CoMi-Annals2012},
\begin{equation}\label{Simon}
\frac{1}{2}\Delta_{f}\left\vert \mathbf{A}\right\vert ^{2}+\left\vert
\mathbf{A}\right\vert ^{2}\left(  \left\vert \mathbf{A}\right\vert
^{2}-1\right)  =\left\vert D\mathbf{A}\right\vert ^{2}.
\end{equation}
Since, by assumption, $|\mathbf{A}|\leq1$ then the strong maximum principle for the
drifted Laplacian yields that either (a) $|\mathbf{A}|<1$ or (b)
$|\mathbf{A}|\equiv1$, on $\Sigma$. In case (a), recalling (\ref{ric_f}), we
deduce that
\[
\mathrm{Ric}_{f}\left(  \partial_{t},\partial_{t}\right)  =\mathrm{Hess}(f)(\partial_{t},\partial_{t})>0\text{ on }\Sigma,
\]
contradicting (\ref{splitting-hessian}). Suppose that (b) holds, namely,
$|\mathbf{A}|\equiv1$. Using again the Simons equation we get that
$\mathbf{A}$ is parallel. We can therefore apply a classification theorem by
Lawson and deduce that $x(\Sigma)$ is a cylindrical product $\mathbb{S}%
_{\sqrt{k}}^{k}\times\mathbb{R}^{m-k}$ with $k=0,...,m$. Since the
self--shrinker is bounded, we conclude that $\Sigma=\mathbb{S}_{\sqrt{m}}^{m}%
$, contradicting the assumption that $\Sigma$ is not compact.
\end{proof}

\subsection{Bounded self--shrinkers with $\limsup |\mathbf{A}|<1$}
In the two previous results we considered global bounds on the norm of  the second fundamental form. The application of the Feller property for $\Delta_f$ in combination with the maximum principle at infinity enable us to prevent the existence of complete, non--compact, bounded self--shrinkers even in the case a pinching condition on $|\mathbf{A}|$ is required at infinity. Recall that the weighted manifold $\Sigma_f$ is said to be Feller if, for some (hence any) smooth domain $\Omega \subset\subset \Sigma_f$ and $\lambda>0$, the minimal solution $h>0$ of the exterior boundary value problem
\[
\left\{
\begin{array}
[c]{ll}%
\Delta_f h=\lambda h & \text{on }\Sigma\setminus \overline{\Omega}\\
h=1 & \text{on }\partial\Omega
\end{array}
\right.
\]
satisfies $h(x) \to 0$ as $x \to \infty$; see \cite{PS-JFA, BPS-RevMatIb}. In particular we obtain the following
\begin{theorem}\label{BndAtInf}
Let $x:\Sigma^m\to\mathbb{B}_{R_0}^{m+1}(0)\subset\mathbb{R}^{m+1}$ be a complete self--shrinker with $\lim_{R\to\infty}\sup_{\Sigma\setminus B_{R}^{\Sigma}}|A|<1$. Then $\Sigma$ is compact.
\end{theorem}

\begin{remark}
\rm{
Suppose that $\Sigma$ is compact. Then  $\Sigma\setminus B_{R}^{\Sigma} = \emptyset$ for $R > \mathrm{diam}(\Sigma)$ and, therefore, $\lim_{R\to\infty}\sup_{\Sigma\setminus B_{R}^{\Sigma}}|A|=-\infty$, proving that the assumption of the theorem is automatically satisfied.
Note also that, from a different perspective, the result states that a complete, non--compact, bounded self--shrinker must satisfy the asymptotic condition
$\lim_{R\to\infty}\sup_{\Sigma\setminus B_{R}^{\Sigma}}|A|\geq 1$.
}
\end{remark}

\begin{proof}
First observe that, since $|\mathbf{A}|\in L^{\infty}(\Sigma)$ and $|\nabla f|=|x^T|\leq |x|<R_{0}$, we know by \eqref{ric_f}, Theorem 7 and Theorem 8 in \cite{BPS-RevMatIb} that $M$ is both stochastically complete and Feller with respect to $\Delta_f$. Furthermore, by \eqref{Simon} and our assumption, we have that $|\mathbf{A}|$ is a bounded nonnegative solution of
\begin{equation}\label{IneqAtInf}
\Delta_f|\mathbf{A}|^2\geq\lambda|\mathbf{A}|^2
\end{equation}
outside a smooth domain $\Omega\subset\subset \Sigma$.
An application of Theorem 2 in \cite{BPS-RevMatIb} permits to deduce that
\begin{equation}\label{LimInfA}
\ |\mathbf{A}|(x)\to0,\quad\mathrm{as}\quad x\to\infty.
\end{equation}
In the matter of this, note that the proof in \cite{BPS-RevMatIb} actually works for nonnegative solutions at infinity of inequalities of the form \eqref{IneqAtInf}.

On the other hand, using the self-shrinker equation, we compute%
\[
\mathrm{Hess}\left(  f\right)  =g-\left\langle
\mathbf{H},\mathbf{A}\right\rangle .
\]
By \eqref{LimInfA} having fixed any ray $\gamma:[0,+\infty)\rightarrow\Sigma$, we have%
\[
\frac{d^{2}}{dt^{2}}\left(  f\circ\gamma\right)
\left(  t\right)  =\mathrm{Hess}\left(  f\right)
\left(  \dot{\gamma},\dot{\gamma}\right)  \geq \frac{1}{2}\text{, for }t>>1.
\]
It follows by integration that $|x|^{2}\rightarrow
+\infty$ along $\gamma$ and, therefore, $x\left(  \Sigma\right)  $ is
unbounded. Contradiction.
\end{proof}

\subsection{Bounded self--shrinkers with $|\mathbf{A}|\in L^{p\geq m}$}

In the next result we switch from $L^{\infty}$ to $L^p$ conditions on the norm of the second fundamental form. In particular we show that complete bounded self--shrinkers with finite total curvature must be compact.
\begin{theorem}
Let $x:\Sigma^{m}  \rightarrow\mathbb{R}^{m+1}$ be a
complete, bounded self--shrinker satisfying $\left\vert \mathbf{A}\right\vert \in L^{p}\left(  d\mathrm{vol}\right),$ for some $p\geq m$. Then $\Sigma$ is compact.
\end{theorem}

\begin{proof}
By contradiction, suppose that $\Sigma$ is complete and non-compact. To illustrate the argument, let us first consider the case $p=m$. Since $f$ is bounded
and $\left\vert \mathbf{H}\right\vert \in L^{m}\left(  \Sigma\right)  $ it is standard to obtain
that $\Sigma$ enjoys the weighted $L^{2}$-Sobolev inequality%
\[
\left(  \int\varphi^{\frac{2m}{m-2}}d\mathrm{vol}_{f}\right)  ^{\frac{m-2}{m}}\leq
S\int\left\vert \nabla\varphi\right\vert ^{2}d\mathrm{vol}_{f},
\]
for some constant $S>0$ and for every $\varphi\in C_{c}^{\infty}\left(
\Sigma\right)  $. Indeed, first we can absorb the mean curvature term in the Sobolev inequality by J.H. Michael and L.M. Simon, \cite{MiSi-CPAM}, outside a large compact set, then,  according to \cite{Ca-Duke}, we can extend the resulting Sobolev inequality to all of $\Sigma$ and, finally, we note that, since $f$ is bounded,%
\[
c^{-1}d\mathrm{vol}_{f}\leq d\mathrm{vol}\leq c\text{ }d\mathrm{vol}_{f}%
\]
for a large enough constant $c>1$.

Now we recall that the second fundamental form of the self--shrinker satisfies
the Simons-type inequality%
\[
\Delta_{f}\left\vert \mathbf{A}\right\vert +\left\vert \mathbf{A}\right\vert
^{3}\geq0.
\]
Since $\left\vert \mathbf{A}\right\vert \in L^{m}\left(  d\mathrm{vol}_{f}\right)  $,
combining the PDE with the weighted Sobolev inequality gives the Anderson-type
decay estimate%
\begin{equation}
\sup_{\Sigma\backslash B_{R}^{\Sigma}\left(  o\right)  }\left\vert
\mathbf{A}\right\vert =o\left( R^{-1}\right)  \text{, as }%
R\rightarrow+\infty. \label{unifest}
\end{equation}
This follows e.g. by adapting to the weighted setting the arguments in
\cite{PiVe-DGA}.

From this uniform estimate it is now standard to get that the immersion $x$ is
proper, thus contradicting the assumption that $x\left(  \Sigma\right)  $ is a
bounded subset of $\mathbb{R}^{m+1}.$ In fact, we have the following general
result that, in the setting of minimal submanifolds of the Euclidean space,
traces back to a paper by M. Anderson, \cite{An-preprint}; see also Remark \ref{rem_tamed} below.

\begin{lemma}\label{lemma_finitetoptype}
Let $x\colon \left(  \Sigma^{m},g\right)  \rightarrow\mathbb{R}^{m+1}$ be a
complete, non-compact hypersurface satisfying (\ref{unifest}). Then $x$ is
proper and $\Sigma$ has finite topological type, i.e., there exists a smooth
compact subset $\Omega\subset\subset\Sigma$ such that $\Sigma\backslash\Omega$
is diffeomorphic to the half-cylinder $\partial\Omega\times\lbrack0,+\infty)$.
\end{lemma}

As a matter of fact, the uniform decay condition (\ref{unifest}) on the second fundamental form, as well as the corresponding structure Lemma,
are even too much strong for the desired conclusion to hold. This is illustrated in the next reasonings where we assume the general condition $p\geq m$.

Again, by contradiction, suppose that $\Sigma$ is complete and non--compact. Since $f$ is
bounded, by the self--shrinker equation we get $\left\vert \mathbf{H}\right\vert \in L^{\infty}$. Whence, we obtain that $\Sigma$ enjoys the weighted $L^{2}$-Sobolev inequality (with potential term)
\[
\left(  \int_{\Sigma}\varphi^{\frac{2m}{m-2}}d\mathrm{vol}_{f}\right)
^{\frac{m-2}{m}}\leq A\int_{\Sigma}\left\vert \nabla\varphi\right\vert
^{2}d\mathrm{vol}_{f}+B\int_{\Sigma}\varphi^{2}d\mathrm{vol}_{f},
\]
for some constants $A,B>0$ and for every $\varphi\in C_{c}^{\infty}\left(
\Sigma\right)  $. Since $\left\vert \mathbf{A}\right\vert $ is a solution of
the semilinear equation%
\[
\Delta_{f}\left\vert \mathbf{A}\right\vert +\left\vert \mathbf{A}\right\vert
^{3}\geq0,
\]
and $\left\vert \mathbf{A}\right\vert \in L^{p}\left(  d\mathrm{vol}%
_{f}\right)  =L^{p}\left(  d\mathrm{vol}\right)  $ for some $p\geq m$, we
deduce that (see e.g. \cite{PiVe-DGA})%
\begin{equation}
\sup_{\Sigma\backslash B_{R}^{\Sigma}}\left\vert \mathbf{A}\right\vert
=o\left(  1\right)  \text{, as }R\rightarrow+\infty.\label{unifest1}%
\end{equation}
Reasoning exactly as in the last part of the proof of Theorem \ref{BndAtInf} this leads to the fact that $x(\Sigma)$ is unbounded, yielding a contradiction.
\end{proof}

\begin{remark}\label{rem_tamed}
\rm{
The decay assumption (\ref{unifest}) in Lemma \ref{lemma_finitetoptype} can be considerably relaxed. This was established in \cite{BJM-CAG} where the authors used the notion of tamed submanifolds. We are grateful to Pacelli Bessa for having pointed out this fact to us.
}
\end{remark}

\section{Self--shrinkers and hyperplanes through the origin} \label{section halfspace}
\subsection{Self--shrinkers in a half--space}
It is reasonable that a complete self--shrinker has a certain homogeneous distribution around $0\in \mathbb{R}^{m+1}$ and, therefore, it should intersect every hyperplane through the origin. For compact self--shrinkers this property is easily verified. In fact, more is true. It was proved in Theorem 7.3 of \cite{WeWa-JDG} that if the distance between two properly immersed self--shrinkers (either compact or not) is realized, then the self--shrinkers must intersect. In particular, a compact self--shrinker must intersect every hyperplane through the origin, as claimed. Moreover, the intersection must be non-tangential by maximum principle considerations. Summarizing, a compact self--shrinker cannot be contained in one of the half--spaces determined by a hyperplane through the origin. Needless to say, exactly the same proof works for a complete self--shrinker with polynomial volume growth because, according to \cite{ChZh-volume}, it is properly immersed. We are going to recover the same conclusion by using more direct and analytic arguments that are suitable for a generalization to the complete, (non--necessarily proper) setting.

\begin{theorem}
\label{th_halfspace-compact}Let $x\colon \Sigma^{m}\rightarrow\mathbb{R}^{m+1}$ be a
compact self--shrinker. Then, for every hyperplane $\Pi$ through the origin of
$\mathbb{R}^{m+1}$, $x\left(  \Sigma\right)  $ cannot be contained in one of the closed halfspaces determined by $\Pi$.
\end{theorem}

\begin{proof}
Recall that, for a self--shrinker,%
\[
\Delta_{f}x=-x.
\]
Therefore, if $\Pi$ has normal equation%
\begin{equation}
\Pi:\text{\quad}L\left(  y\right)  :=\sum_{j=1}^{m+1}a_{j}y^{j}%
=0,\label{hyperplane}%
\end{equation}
we have that the self--shrinker satisfies also%
\begin{equation}
\Delta_{f}L\left(  x\right)  =-L\left(  x\right)  .\label{eq-L(x)}%
\end{equation}
Whence, it follows easily that $x\left(  \Sigma\right)  $ cannot be contained
in one of the closed half-spaces determined by $\Pi$. Indeed, otherwise, we would
have that either $L\left(  x\right)  \geq0$ or $L\left(  x\right)  \leq0$.
Without loss of generality, suppose that $L\left(  x\right)  \geq0$. Then, by
the above equation, $L\left(  x\right)  $ would be an $f$-superharmonic
function on the compact manifold $\Sigma$. By the maximum principle
\ $L\equiv\mathrm{const}$ and by equation (\ref{eq-L(x)}) $L\equiv0$. This
means that $x\left(  \Sigma\right)  \subseteq\Pi$ and, by geodesic
completeness, $x\left(  \Sigma\right)  =\Pi$. This is clearly impossible
because \ $\Sigma$ is compact.
\end{proof}

A similar conclusion can be obtained for complete self--shrinkers  $x\colon \Sigma^{m}\rightarrow\mathbb{R}^{m+1}$
with a controlled extrinsic geometry. By way of example, suppose that
\begin{equation}
\left\vert x \right\vert + \left\vert \mathbf{A}\left(
p\right)  \right\vert \leq \sqrt{1+r\left(  p\right)^2}
,\label{growth x and A}%
\end{equation}
where $r\left(  p\right)  =d_{\Sigma}\left(  p,o\right)  $.
Then, for every hyperplane $\Pi$ through the origin, if $x\left(
\Sigma\right)  $ lies on one side of $\Pi$, then%
\[
\mathrm{dist}_{\mathbb{R}^{m+1}}\left(  \Pi,x\left(  \Sigma\right)  \right)  =0
\]
and the distance is not attained, unless $x\left(  \Sigma\right)  =\Pi$.

Indeed, note that, in light of (\ref{ric_f}),
condition (\ref{growth x and A}) implies
\[
Ric_{f}\geq-C(1+r^2)  \text{,\quad}\left\vert \nabla f\right\vert
=|x^{T}|\leq \sqrt{1+r^2}  .
\]
Then, according to Corollary 5.3 in \cite{PiRiRiSe-Annali}, for every $u\in
C^{2}\left(  \Sigma\right)  $ with $\inf_{\Sigma}u=u_{\ast}>- \infty$ \ there
exists a sequence $\left\{  p_{n}\right\}  \subset\Sigma$ \ along which%
\[
u\left(  p_{n}\right)  <u_{\ast}+\frac{1}{n}\text{,\quad}\left\vert \nabla
u\right\vert \left(  p_{n}\right)  <\frac{1}{n}\text{,\quad}\Delta_{f}u\left(
p_{n}\right)  >-\frac{1}{n}.
\]
Now, as in the compact case, if $x\left(  \Sigma\right)  $ lies on one side of
$\Pi$, we can assume that $L\left(  x\right)  \geq0$ where $L\left(  y\right)
$ is defined in (\ref{hyperplane}). Evaluating (\ref{eq-L(x)}) along $\left\{
p_{n}\right\}  $ we deduce that $\inf_{\Sigma}L\left(  x\right)  =0$, as
desired. The second conclusion is a consequence of the strong
minimum principle for positive super-solutions of $\Delta_{f}+1$.
\medskip

In the next theorem we point out natural geometric conditions that permit to recover the full
conclusion of the compact case.
\begin{theorem}
\label{th_halfspace-complete}
Let $x\colon \Sigma^{m}\rightarrow\mathbb{R}^{m+1}$ be
a complete, non-compact self--shrinker. Assume that either one of the following assumptions is satisfied:
\begin{itemize}
\item[(a)] $\Sigma$ has (extrinsic) polynomial volume growth.
\item[(b)] $\mathrm{vol}_f(B_R^{\Sigma})=O(R^2)$ as $R \to \infty$.
\item[(c)] $\left\vert \mathbf{A}\right\vert^2 \in L^{p} \left(
d\mathrm{vol}_{f}\right)$ and $\left\vert \mathbf{A}\right\vert ^{2}\leq 1+\frac{1}{p}$,
for some $p> 1$.
\end{itemize}
\medskip
Then, for every hyperplane $\Pi$ through the origin, if $x\left(  \Sigma\right)  $ lies on one side of $\Pi$, then $x\left(
\Sigma\right)  =\Pi$.
\end{theorem}

\begin{proof} We shall use extensively the notation introduced so far. In particular, the hyperplane $\Pi$
is described by the normal equation (\ref{hyperplane}) and the function $L(x)$ satisfies equation (\ref{eq-L(x)}).

Assume we are in the assumptions of \textbf{(a)}. Since $\Sigma$ has polynomial volume growth, then
$\mathrm{vol}_{f}\left(  \Sigma\right)  <+\infty$ \ and $\Sigma_{f}$ is parabolic with
respect to the drifted Laplacian $\Delta_{f}$. Using the above notation,
assume without loss of generality that $L\left(  x\right)  \geq0$. By
equation (\ref{eq-L(x)}) we see that $\ L\left(  x\right)  \geq0$ is
$f$-superharmonic, hence it is constant by $f$--parabolicity. The desired
conclusion now follows as in the proof of Theorem \ref{th_halfspace-compact}. Case \textbf{(b)} is completely similar.
Assume now that the assumptions in \textbf{(c)} are satisfied. Let
$x\left(  \Sigma\right)  \neq\Pi$ and, by contradiction, suppose that
$x\left(  \Sigma\right)  $ is contained in a half-space determined by $\Pi$.
Then, by the strong minimum principle, we can assume that $L\left(  x\right)
>0$ is a solution of%
\[
\Delta_{f}L+L=0.
\]
Since%
\[
p\left(  \left\vert \mathbf{A}\right\vert ^{2}-1\right)  \leq 1,
\]
for some $p>1$, we obtain%
\[
\Delta_{f}L+p\left(  \left\vert \mathbf{A}\right\vert ^{2}-1\right)
L\leq0.
\]
Combining this latter with the Simons--type inequality%
\[
\left\vert \mathbf{A}\right\vert \left\{  \Delta_{f}\left\vert \mathbf{A}%
\right\vert +\left\vert \mathbf{A}\right\vert \left(  \left\vert
\mathbf{A}\right\vert ^{2}-1\right)  \right\}  \geq\left\vert D\mathbf{A}%
\right\vert ^{2}-\left\vert \nabla\left\vert \mathbf{A}\right\vert \right\vert
^{2}\geq0,
\]
and applying Theorem 8 in \cite{Ri-self} we conclude that either $\left\vert
\mathbf{A}\right\vert \equiv1$ or $\left\vert \mathbf{A}\right\vert \equiv0$.
Using this information into the Simons--type equality%
\[
\frac{1}{2}\Delta_{f}\left\vert \mathbf{A}\right\vert ^{2}+\left\vert
\mathbf{A}\right\vert ^{2}\left(  \left\vert \mathbf{A}\right\vert
^{2}-1\right)  =\left\vert D\mathbf{A}\right\vert ^{2}%
\]
gives that $\left\vert D\mathbf{A}\right\vert \equiv0$ and by Lawson
classification theorem $x(  \Sigma)  =\mathbb{S}^{k}_{  \sqrt{k}}
\times\mathbb{R}^{m-k}$, with $0\leq k\leq m$. Since $x\left(  \Sigma\right)
$ must lie on one side of $\Pi$ we necessarily have $k=0$, i.e., $x\left(  \Sigma\right)
=\Pi$, contradiction.
\end{proof}

\subsection{Bottom of the spectrum of the drifted Laplacian}
Once we have understood that, to a certain extent, complete self--shrinkers
intersect transversally a hyperplane through the origin, we are going to deduce
spectral information on the drifted Laplacian whenever the intersection is compact,
and some (extrinsic) volume growth condition is satisfied.

The intuition for the general result contained in Theorem \ref{th_spectral} relies on the following two examples.
Recall that, by definition, the bottom of the spectrum of $-\Delta_f$ on a domain $\Omega \subseteq \Sigma$, with
Dirichlet boundary conditions, is defined by
\[
\lambda_{1}(  -\Delta_{f}^\Omega)=\inf_{v\in C^{\infty}_{c}(\Omega)\setminus\{0\}}\frac{\int_{\Omega}{|\nabla v|^2 d\mathrm{vol}_f}}{\int_{\Omega}v^2 d\mathrm{vol}_f}.
\]
The bottom of the spectrum $\lambda_1$ is an eigenvalue of $-\Delta_f$ if there exists a function $u\in \mathrm{Dom}(-\Delta_{f}^{\Omega})$ such that
\[
-\Delta_f u = \lambda_1 u \text{ on }\Omega,
\]
where
\[
\mathrm{Dom}(-\Delta_{f}^{\Omega})=\{u \in W^{1,2}_0(\Omega,d\mathrm{vol}_f):\Delta_f u \in L^2(\Omega,d\mathrm{vol}_f)\}
\]
is the domain of (the Friedrichs extension of) $-\Delta_{f}$ originally defined on $C^{\infty}_{c}(\Omega)$. For future purposes,
we also recall that if $\partial \Omega$ is compact then,
\[
u\in W^{1,2}(\Omega,d\mathrm{vol}_f) \text{ and }u=0 \text{ on }\partial{\Omega} \Rightarrow { } u \in W^{1,2}_0(\Omega,d\mathrm{vol}_f).
\]
Indeed, the interesting case occurs when $\Omega$ is non--compact, i.e., an exterior domain, in the complete manifold $\Sigma$. Let $0\leq \phi_R \leq 1$ be the standard family of cut--off functions supported in the ball $B^{\Sigma}_{2R}$, satisfying $\phi_{R}=1$ on $B^{\Sigma}_R$ and such that $|\nabla \phi_{R}|\leq 2/R$. Then, $u_R=u\phi_R \in W^{1,2}_{0}(\Omega)$ and it is easy to verify that
$u_{R} \to u$ in $W^{1,2}(\Omega,d\mathrm{vol}_f)$, as $R \to \infty$.
\begin{example}
\rm{Consider the self--shrinker sphere $\mathbb{S}_{\sqrt{m}}^{m}$. Then, each hyperplane
$\Pi$ through the origin divides $\mathbb{S}_{\sqrt{m}}^{m}$ into half--spheres
isometric to $^{+}\mathbb{S}_{\sqrt{m}}^{m}=\mathbb{S}_{\sqrt{m}}^{m}\cap\left\{  y^{m+1}>0\right\}
$. Since $f\left(  x\right)  \equiv m/2$, it holds%
\[
\lambda_{1}(  -\Delta_{f}^{^{+}\mathbb{S}_{\sqrt{m}}^{m}})  =\lambda_{1}(
-\Delta^{^{+}\mathbb{S}_{\sqrt{m}}^{m}})  =\frac{1}{m}\lambda_{1}(
-\Delta^{^{+}\mathbb{S}_{1}^{m}})  =1;
\]
see e.g. \cite{Ch_Eigenvalues}.}
\end{example}

\begin{example}
\rm{Consider the self--shrinker cylinder $\mathcal{C}=\mathbb{S}_{\sqrt{m-1}}%
^{m-1}\times\mathbb{R}$. Then the hyperplane $\Pi=\left\{  y^{m+1}=0\right\}
$ intersects $\mathcal{C}$ along the sphere $\mathbb{S}_{\sqrt{m-1}}%
^{m-1}$ and divides $\mathcal{C}$ into two half--cylinders
isometric to $\mathcal{C}_{+}=\mathbb{S}_{\sqrt{m-1}}^{m-1}\times\mathbb{R}%
_{+}$. These are the ends of $\Sigma$. We claim that%
\[
\lambda_{1}(- \Delta_{f}^{\mathcal{C}_{+}})  =1.
\]
Indeed, since%
\[
f=\frac{|x|^{2}}{2}=\frac{m-1}{2}+\frac{x_{m+1}^{2}}{2},
\]
we have the decomposition%
\[
\Delta_{f}^{\mathcal{C}_{+}}=\Delta^{\mathbb{S}_{\sqrt{m-1}}^{m-1}}%
+\Delta_{t^{2}/2}^{\mathbb{R}_{+}}%
\]
and, therefore,%
\begin{align*}
\lambda_{1}( - \Delta_{f}^{\mathcal{C}_{+}})   &
=\lambda_{1}(-  \Delta^{\mathbb{S}_{\sqrt{m-1}}^{m-1}})  +\lambda_{1}
(-\Delta_{t^{2}/2}^{\mathbb{R}_{+}}) \\
&  =0+\lambda_{1}(-  \Delta_{t^{2}/2}^{\mathbb{R}_{+}})  .
\end{align*}
Now, the Ornstein--Uhlenbeck operator $\Delta_{t^{2}/2}^{\mathbb{R}}$ on
$(  \mathbb{R}_{+},e^{-t^{2}/2}dt)  $ satisfies%
\[
\lambda_{1}(-  \Delta_{t^{2}/2}^{\mathbb{R}_{+}})  =1.
\]
See e.g. the lecture notes \cite{Sj-lectures} for the basic theory and more advanced topics on the Ornstein--Uhlenbeck operator and its semigroup. Indeed, $u\left(  t\right)  =t$ is a smooth, positive function on $\mathbb{R}_{+}$
satisfying%
\begin{equation}\label{OU-eigenvalue}
\Delta_{t^{2}/2}^{\mathbb{R}_{+}}u=u^{\prime\prime}-tu^{\prime}=-u
\end{equation}
so that, by (the weighted version of) Barta's theorem,
\[
\lambda_{1}(  \Delta_{t^{2}/2}^{\mathbb{R}_{+}})  \geq \inf_{\mathbb{R}_{+}} \frac
{-\Delta_{t^{2}/2}^{\mathbb{R}_{+}}u}{u}=1.
\]
On the other hand, $u\in W^{1,2}(
\mathbb{R}_{+},e^{-t^{2}/2}dt)  $, therefore, by (\ref{OU-eigenvalue}), $\Delta_{t^{2}/2}^{\mathbb{R}_{+}}u \in L^2(\mathbb{R}_{+},e^{-t^{2}/2}dt)  $. Furthermore,  $u\left(  0\right)  =0$. It follows that $u \in \mathrm{Dom}(-\Delta_{f}^{\mathbb{R}_{+}})$ is also a Dirichlet
eigenfunction of the Ornstein--Uhlenbeck operator on $\mathbb{R}_{+}$.}
\end{example}

Abstracting from the previous examples we are now ready to state the following general result.

\begin{theorem} \label{th_spectral}
Let $i:\Sigma^{m}\hookrightarrow\mathbb{R}^{m+1}$ be a complete, embedded
self--shrinker. Assume that, for some hyperplane $\Pi\approx\mathbb{R}^{m}$
through the origin, $\Sigma\cap\Pi=K$ is a compact $\left(  m-1\right)
$--dimensional submanifold. Then:

\begin{enumerate}
\item[(a)] for every connected component $\Sigma_{1}$ of $\Sigma\backslash K$
(which is an open submanifold $\Sigma_{1}\subset\Sigma$ with $\partial
\Sigma_{1}\subseteq K$) it holds%
\[
\lambda_{1}(  -\Delta_{f}^{\Sigma_{1}})  \geq1.
\]

\item[(b)] If either $\Sigma$ is compact or $\Sigma$ has only one end, then
there exists a bounded connected component $\Sigma_{2}$ of $\Sigma\backslash
K$ such that%
\[
\lambda_{1}(  -\Delta_{f}^{\Sigma_{2}})  =1.
\]

\item[(c)] If $\Sigma_{3}$ is an end of $\Sigma$ with respect to $K$ with extrinsic volume growth
\begin{equation}\label{ExtrVolSubExpQuadrEnd}
\mathrm{vol}\left(  \Sigma_{3}\cap\mathbb{B}_{R}^{m+1}\right)  =O(
e^{\alpha R^{2}})  \text{, as }R\rightarrow+\infty,
\end{equation}
for some $0\leq\alpha<1/2$, then
\[
\lambda_{1}(  -\Delta_{f}^{\Sigma_{3}})  =1.
\]

\end{enumerate}
\end{theorem}

\begin{remark}
\rm{The conclusion in (a) holds regardless of the fact that the intersection $K$ is compact.}
\end{remark}
\begin{remark}
\rm{Note that condition (\ref{ExtrVolSubExpQuadrEnd}) in (c) is actually equivalent to the (only apparently less general) polynomial volume growth condition.
Indeed, it is easy to see that (\ref{ExtrVolSubExpQuadrEnd}) implies that $\mathrm{vol}_{f}(\Sigma_3 )< +\infty$ (see Lemma \ref{VolGrowths} below)
and minor changes to the proofs of Theorem 2.1 and Theorem 4.1 in \cite{ChZh-volume} show that the equivalences in \cite{ChZh-volume} can be localized to a given end. In particular, under assumption \eqref{ExtrVolSubExpQuadrEnd}, $\Sigma_3$ is proper and of extrinsic polynomial (Euclidean) volume growth.
For the sake of completeness we sketch out here the proof of the fact that a properly immersed end has Euclidean volume growth. The proofs of the remaining implications can be easily adapted from the original ones. Suppose that $\tilde{\Sigma}$ is a properly immersed end of a complete noncompact self-shrinker $x:\Sigma^m\to\mathbb{R}^{m+1}$. To prove that $\tilde{\Sigma}$ must have Euclidean extrinsic volume growth observe that, since $\partial \tilde{\Sigma}$ is compact and properly immersed we can find a regular value $r_0$ such that $\left\{p\in\tilde{\Sigma}:\,|x(p)|=r_0\right\}$ does not intersect $\partial \tilde{\Sigma}$. Then we can define for $r>r_0$ the set $D_r:=\left\{p\in\tilde{\Sigma}:\,r_0<|x(p)|<r\right\}$. Since the immersion is proper, letting $h=\frac{|x|^2}{4}$, we can define for $t>0$, $r>r_0$,
\[
\ I(t)=\frac{1}{t^m}\int_{\overline{D}_r}e^{-\frac{h}{t}}d\rm{vol}.
\]
Since on a self-shrinker
\begin{eqnarray*}
\left|\nabla h\right|^2-h\leq 0\\
\Delta_{h}h+h=m,
\end{eqnarray*}
we gain that, if $t\geq1$,
\begin{equation*}
I^{\prime}(t)\leq-t^{-m-1}\int_{\overline{D}_r}\mathrm{div}\left(e^{-\frac{h}{t}}\nabla h\right).
\end{equation*}
At a regular value $r$ of $|x|$, for $t\geq 1$, by Stokes' Theorem we have thus
\begin{eqnarray*}
I^{\prime}(t)&\leq-t^{-m-1}\Bigg[&\left.\int_{\left\{|x|=r\right\}}\left\langle e^{-\frac{h}{t}}\nabla h, \frac{\nabla h}{|\nabla h|}\right\rangle d\mathrm{vol}\right.\\
&&\left.-\int_{\left\{|x|=r_{0}\right\}}\left\langle e^{-\frac{h}{t}}\nabla h, \frac{\nabla h}{|\nabla h|}\right\rangle d\mathrm{vol}\right]\\
&\leq& t^{-m-1}\int_{\left\{|x|=r_0\right\}}e^{-\frac{h}{t}}|\nabla h|d\mathrm{vol}.
\end{eqnarray*}
Integrating on $[1,r^2]$, with $r^2>r_0^2\geq1$, and elaborating, we get
\begin{equation}\label{euclvolend}
e^{-\frac{1}{4}}r^{-2m}\int_{\overline{D}_r} d\mathrm{vol}\leq \int_{\overline{D}_r}e^{-h}d\mathrm{vol}+\int_{1}^{r^2}t^{-m-1}e^{-\frac{r_0^2}{2t}}dt\int_{\left\{|x|=r_0\right\}}|\nabla h|d\mathrm{vol}.
\end{equation}
Proceeding now as in \cite{ChZh-volume} we can conclude that, for any positive integer $N$, we have
\begin{eqnarray*}
\int_{\overline{D}_{r+N}}e^{-h}d\mathrm{vol}\leq \left[\prod_{i=0}^{N}\frac{1}{1-e^{-(r+i)}}\right]&&\left[\int_{\overline{D}_{r-1}}e^{-h}d\mathrm{vol}\right.\\&&\left.+e^{-r}\int_{1}^{r^2}t^{-m-1}e^{-\frac{r_0^2}{2t}}\frac{r_0}{2}\mathrm{vol}_{m-1}\left(\left\{|x|=r_0\right\}\right)\right].
\end{eqnarray*}
This implies that $\int_{\tilde{\Sigma}}e^{-h}d\mathrm{vol}<+\infty$ and the desired Euclidean extrinsic volume growth of $\tilde{\Sigma}$ follows from (\ref{euclvolend}).
}
\end{remark}

\begin{proof}[Proof (of Theorem \ref{th_spectral})]
Let $\Pi$ be represented by the normal equation%
\[
\Pi:L\left(  y\right)  :=\sum_{j=1}^{m+1}a_{j}y^{j}=0.
\]
Recall that, for every self--shrinker,%
\[
\Delta_{f}x=-x.
\]
It follows that%
\[
\Delta_{f}L\left(  x\right)  +L\left(  x\right)  =0\text{, on }\Sigma.
\]
In particular, this equation holds on $\Sigma_{1}$. Moreover, since
$\Sigma_{1}$ is contained in one of the open halfspaces determined by $\Pi$,
then either $L<0$ or $L>0$ on $\Sigma_{1}$. Thus, up to changing the sign of
$L$, we can assume $L>0$ and using (the weighted version of) Barta's theorem
we deduce%
\[
\lambda_{1}(  -\Delta_{f}^{\Sigma_{1}})  \geq \inf_{\Sigma_1}\frac{-\Delta_{f}L}%
{L}=1.
\]
This proves (a).

Suppose now that $\Sigma$ is non--compact and has only one end. We claim that there
exists a compact connected component $\Sigma_{2}$ of $\Sigma\backslash K$. In
this case, since $L=0$ on $\partial\Sigma_{2}\subseteq K$, we deduce that $L$
is an eigenfunction of $\Delta_{f}^{\Sigma_{2}}$ corresponding to eigenvalue
$+1$. When combined with (a) this clearly implies that $\lambda_{1}%
(-\Delta_{f}^{\Sigma_{2}})=1$, completing the proof of (b). To prove the
claim, we first observe that $\ \Sigma\backslash K$ cannot be connected.
Indeed, by contradiction, suppose the contrary. Then $\Sigma$ must be
contained in one of the closed half--spaces determined by $\Pi$ and intersects
$\Pi$ tangentially along $K$. Without loss of generality, we can assume that
$L\left(  x\right)  \geq0$ on $\Sigma$ and $L\left(  x\right)  =0$ on $K$.
Since $\Delta_{f}L\left(  x\right)  =-L\left(  x\right)  \leq0$ on $\Sigma$,
by the strong minimum principle we get $L\left(  x\right)  \equiv0$ on
$\Sigma$, i.e., $\Sigma\subseteq\Pi$. Actually, $\Sigma=\Pi$ by geodesic
completeness and this clearly prevents $K=\Sigma\cap\Pi$ to be compact,
contradiction. Thus, $\Sigma\backslash K$ has at least two connected
components. Since we are assuming that $\Sigma$ has one end, at most one of
them can be unbounded. We
therefore find a bounded component $\Sigma_{2}\subseteq\Sigma$ of $\Sigma\setminus K$, as claimed.

It remains to prove (c). The argument is completely similar to the above.
According to (a), $\lambda_{1}(-\Delta_{f}^{\Sigma_{3}})\geq1$ and $L\left(
x\right)  \geq0$ is a solution of%
\[
\left\{
\begin{array}
[c]{l}%
\Delta_{f}L\left(  x\right)  +L\left(  x\right)  =0\text{, on }\Sigma_{3}\\
L=0\text{, on }\partial\Sigma_{3}\subseteq K.
\end{array}
\right.
\]
To conclude that, in fact, $\lambda_{1}(-\Delta_{f}^{\Sigma_{3}})=1$ it
suffices to show that $L \in \mathrm{Dom}(\Sigma_3)$. Since  $L=0$ on the compact boundary  $\partial \Sigma_3$, we have to show that $L \in W^{1,2}\left(\Sigma_3, d\mathrm{vol}_{f}\right)  $. To this aim, we simply note that
\[
\frac{\left\vert L\left(  x\right)  \right\vert }{\sqrt{\sum a_{j}^{2}}%
}=\mathrm{dist}_{\mathbb{R}^{m+1}}\left(  x,\Pi\right)  \leq d_{\mathbb{R}%
^{m+1}}\left(  x,0\right)  = |x|,
\]
and
\[
\frac{|\nabla L\left(  x\right)| }{\sqrt{\sum a_{j}^{2}}}
\leq 1.
\]
Therefore, we can apply the next trivial lemma. This proves (c) and completes the proof of the theorem.
\end{proof}

\begin{lemma}\label{VolGrowths}
Let $x\colon \Sigma^{m}\rightarrow\mathbb{R}^{m+1}$ be any hypersurface satisfying%
\[
\mathrm{vol}\left(  \Sigma\cap\mathbb{B}^{m+1}_{R}\right)  =O\left(  e^{\alpha
R^2}\right)  \text{, as }R\rightarrow+\infty,
\]
for some $0\leq\alpha<1/2$. Then, for every polynomial $\mathcal{P}\left(
t\right)  $ and for every $0\leq\beta<1/2-\alpha$,%
\[
\mathcal{P}\left(  \left\vert x\right\vert \right)  e^{\beta\left\vert
x\right\vert ^{2}}\in L^{1}\left(  d\mathrm{vol}_{f}\right)  .
\]

\end{lemma}

\begin{proof}
Note that, by assumption, there exists $t>1$ such that%
\[
\frac{1}{2}-t^2 \alpha-\beta>0.
\]
Now, we simply compute%
\begin{align*}
\int_{\Sigma}\left\vert x\right\vert ^{p}e^{\beta\left\vert x\right\vert ^{2}%
}d\mathrm{vol}_{f} &  =\int_{\Sigma}\left\vert x\right\vert ^{p}e^{-\left(
\frac{1}{2}-\beta\right)  \left\vert x\right\vert ^{2}}d\mathrm{vol}\\
&  =C_{1}+C_{2}\sum_{n=0}^{+\infty}\int_{\Sigma\cap\left(  \mathbb{B}%
_{t^{n+1}}^{m+1}\backslash\mathbb{B}_{t^{n}}^{m+1}\right)  }\left\vert
x\right\vert ^{p}e^{-\left(  \frac{1}{2}-\beta\right)  \left\vert x\right\vert
^{2}}d\mathrm{vol}\\
&  \leq C_{1}+C_{2}\sum_{n=0}^{+\infty}t^{pn+p}e^{-\left(  \frac{1}{2}%
-\beta\right)  t^{2n}}\mathrm{vol}\left(  \Sigma\cap\mathbb{B}_{t^{n+1}%
}^{m+1}\right)  \\
&  \leq C_{1}+C_{2}\sum_{n=0}^{+\infty}t^{pn+p}e^{-\left(  \frac{1}{2}%
-t^{2}\alpha-\beta\right)  t^{2n}}\\
&  <+\infty.
\end{align*}

\end{proof}

\end{document}